\documentclass[11pt,a4paper]{amsart} 
\usepackage{amssymb} 
\usepackage{amsthm, amsmath}  
\usepackage{mathrsfs}
\usepackage{color} 
\usepackage{tikz} 
\usetikzlibrary{positioning}
\usepackage{subfig}
\usepackage{multirow}
\usepackage{ifthen}

\newtheorem{theorem}{Theorem}[section]
\newtheorem{lemma}[theorem]{Lemma}
\newtheorem{corollary}[theorem]{Corollary}
\newtheorem{proposition}[theorem]{Proposition}

\newtheorem{observation}[theorem]{Observation}

\theoremstyle{definition}

\newtheorem{remark}[theorem]{Remark}    
\newtheorem{example}[theorem]{Example}

\usepackage{mathrsfs}
\usetikzlibrary{positioning}
\usepackage{multirow}
\usepackage{ifthen}
\usepackage{float}

\newcommand{\hil}{{\mathbb H}}
\newcommand{\R}{{\mathbb R}}
\newcommand{\C}{{\mathbb C}}
\newcommand{\Rho}{{\rm P}}

\newcommand{\msr}{{\rm msr}}
\newcommand{\rank}{{\rm rank}}

\newcommand{\CF}{{\mathcal F}}
\newcommand{\SL}{{\mathscr L}}
\newcommand{\CH}{{\mathcal H}}

\newcommand{\bmat}{\left[ \ \begin{matrix}}
\newcommand{\emat}{\end{matrix} \ \right]}
\newcommand{\bpar}{\left\{ \ \begin{matrix}}
\newcommand{\epar}{\end{matrix} \ \right\}}
\newcommand{\eparr}{\end{matrix} \ \right.}
\newcommand{\cp}{\, \Box\,}
\newlength{\mysmnodesz}
\newlength{\mynmnodesz}
\newcommand{\spanscale}{1}

\begin{document} 

\title{Tight Frame Graphs Arising as Line Graphs}
\author[V. Furst]{Veronika~Furst}
\address{Veronika Furst, Department of Mathematics, Fort Lewis College, Durango, CO  81301, USA}
\email{furst\_v@fortlewis.edu}
\author[H. Grotts]{Howard~Grotts}
\address{Howard Grotts, Department of Mathematics, Fort Lewis College, Durango, CO  81301, USA}
\email{hbgrotts@live.com}

\begin{abstract}
Dual multiplicity graphs are those simple, undirected graphs that have a weighted Hermitian adjacency matrix with only two distinct eigenvalues.  From the point of view of frame theory, their characterization can be restated as which graphs have a representation by a tight frame.  In this paper, we classify certain line graphs that are tight frame graphs and improve a previous result on the embedding of frame graphs in tight frame graphs.
\end{abstract} 

\maketitle

\section{Introduction}

Much recent interest has concerned the inverse eigenvalue problem of a graph, namely, to determine all possible spectra of real symmetric (Hermitian) matrices whose off-diagonal pattern of zero/non-zero entries is given by the adjacencies of a graph (see \cite{BBF20} and the references therein).  The special case of dual multiplicity graphs, or graphs that permit two distinct eigenvalues, has been thoroughly investigated (see \cite{AFM19}, \cite{AAC13},  \cite{BHP18}, and \cite{CGM14}).  

In \cite{AN18} this problem was reintroduced from the point of view of frame theory.  Frames in finite-dimensional spaces have received much attention from both pure and applied mathematicians and constitute a vast literature; see \cite{CKP13} for a thorough introduction or \cite{HKL07} for a development at the undergraduate level.  In finite-dimensional spaces, a finite frame is a sequence of vectors whose span is the whole space; the redundancy of frames is their key advantage over orthonormal bases in applications such as image processing or data transmission.  Tight frames \cite {W18} are especially important since they provide (non-unique) reconstruction of vectors similarly to orthonormal bases, without the requirement of orthogonality (or linear independence).  
Every simple graph is a frame graph, that is, has a vector representation by a frame $\{f_1, \ldots, f_n\}$ with $\langle f_i, f_j \rangle \neq 0$ if and only if the vertices represented by $f_i$ and $f_j$ are connected by an edge.  In the real case, this is also known as a faithful orthogonal representation, dating back to \cite{L79}.  Dual multiplicity graphs are those which have a representation by a tight frame, and classifying them is a difficult question.  

In this paper, we continue the line of investigation of \cite{AN18}, applying frame theoretic tools to a graph theoretic question.  In the next section, we provide definitions and preliminary material from both graph theory and frame theory, concluding with some key connections.  In Section 3, we focus on tight frames and establish an improved result on frame graphs being embedded as induced subgraphs of tight frame graphs.  Although a straightforward consequence of Proposition 2.1 of \cite{CFM13}, Corollary \ref{min} yields embeddings that are minimal in terms of the number of additional vectors required; we contrast our construction with the non-minimal embedding of Theorem 4.2 of \cite{CGM14}.  Section 4 contains our main contribution, the natural recognition of frame graphs as line graphs and the classification of certain line graphs as tight frame graphs.  Using what we refer to as the ``Laplacian method," we establish the line graph of the complete graph as a tight frame graph (Theorem \ref{LKn}) and give a new, constructive proof of the tightness of the complete graph (Theorem \ref{Sn1Kn}).  We illustrate relationships between line graphs and root graphs in terms of whether or not they are tight frame graphs and end by pointing out some limitations of this approach.

\section{Background}

A \textit{graph} $\Gamma=(V,E)$ is an ordered pair, where $V=V(\Gamma)=\{v_1,\hdots, v_n \}$ is a non-empty set of {\em vertices} and $E=E(\Gamma)$ is a set of unordered pairs of distinct vertices, called {\em edges} (or {\em lines}) connecting vertices. Here we consider only undirected simple graphs, without loops, or edges connecting a vertex to itself, and without multiple edges between a pair of vertices. We say two vertices $v_i, v_j$ are \textit{adjacent} (or are {\em neighbors}) if $\{v_i, v_j \} \in E$, and two edges are \textit{incident} if they share an endpoint.  The \textit{degree} of a vertex $v$, denoted $\deg(v)$, is the number of vertices adjacent to $v$.  The {\it order} of a graph is $|V|=n$ and its {\it size} is $|E|$. 

A \textit{trail} is a sequence of distinct, incident edges connecting a sequence of vertices. If the vertices in a trail are distinct, the associated graph is a \textit{path}, and if a trail begins and ends at the same vertex but all other vertices are distinct, the graph is a \textit{cycle}. The path and cycle on $n$ vertices are written $P_n$ and $C_n$, respectively. An \textit{induced subgraph} is a subset of vertices of a graph and all edges whose endpoints are both in that subset. If a graph has at least one induced path connecting any two vertices, the graph is {\it connected}. If the deletion of an edge from such a graph results in a disconnected graph, that edge is known as a {\it bridge}.
\begin{figure}[h]
\centering
    \begin{tikzpicture}
    \begin{scope}[every node/.style={circle,draw,inner sep = 0pt, minimum size=4.5mm}]
        \node (v1) at (1,1.3) {\scriptsize{$v_3$}};
        \node (v2) at (0,.65) {\scriptsize{$v_1$}};
        \node (v3) at (2,.65) {\scriptsize{$v_2$}};
        \node (v4) at (1,0) {\scriptsize{$v_4$}};
    \end{scope}
    \path (v1) edge[left] node {} (v2);
    \path (v1) edge[right] node {} (v3);
    \path (v1) edge[right] node {} (v4);
    \path (v2) edge[below left] node {} (v4);
    \path (v3) edge[below right] node {} (v4);
    \end{tikzpicture}
\qquad
\qquad
    \begin{tikzpicture}
    \begin{scope}[every/.style={circle,draw, minimum size=\mynmnodesz}]
        \node[fill=black,every] (v1) at (1,1.3) {};
        \node[fill=black,every] (v2) at (0,.65) {};
        \node[fill=black,every] (v3) at (2,.65) {};
        \node[every] (v4) at (1,0) {};
       
    \end{scope}
    \path (v1) edge[line width=1.5pt] node {} (v2);
    \path (v1) edge[line width=1.5pt] node {} (v3);
    \path (v1) edge node {} (v4);
    \path (v2) edge node {} (v4);
    \path (v3) edge node {} (v4);
    \end{tikzpicture}
\qquad
\qquad
    \begin{tikzpicture}
    \begin{scope}[every node/.style={circle,draw, minimum size=\mynmnodesz}]
        \node[fill=black] (v1) at (1,1.3) {};
        \node[fill=black] (v2) at (0,.65) {};
        \node (v3) at (2,.65) {};
        \node[fill=black] (v4) at (1,0) {};
    \end{scope}
    \path (v1) edge[line width=1.5pt] node {} (v2);
    \path (v1) edge node {} (v3);
    \path (v1) edge[line width=1.5pt] node {} (v4);
    \path (v2) edge[line width=1.5pt] node {} (v4);
    \path (v3) edge node {} (v4);
    \end{tikzpicture}
\qquad
    \caption{The diamond graph is a connected, bridgeless graph of order $4$ and size $5$; $P_3$ and $C_3$ as induced subgraphs}
    \label{fig1}
  \end{figure}
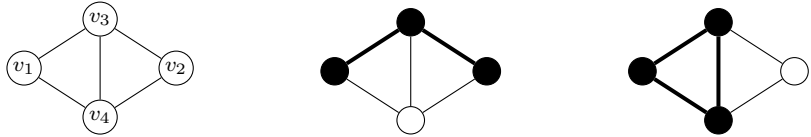
  
The \textit{complete graph} $K_n$ on $n$ vertices has $\{v_i,v_j\}\in E$ for all $i \neq j$. Equivalently, the degree of every vertex of $K_n$ is $n-1$. The \textit{star graph} $S_n$ on $n$ vertices has one vertex of degree $n-1$ and $n-1$ vertices of degree one.  The \textit{complete bipartite graph} $K_{m,n}$ has vertex set $V = V_1\cup V_2$ ($|V_1|=m, |V_2|=n$) and edge set $E = \{ \{v_1, v_2\}: v_1 \in V_1 \mbox{ and } v_2 \in V_2\}$.  Note that $S_n$ is the same graph as $K_{1, n-1}$.  The {\em Cartesian product} of graphs $\Gamma$ and $\Lambda$, denoted $\Gamma \cp \Lambda$, is the graph with vertex set $V(\Gamma) \times V(\Lambda)$, such that vertices $(u,u')$ and $(v,v')$ are adjacent in $\Gamma \cp \Lambda$ if and only if either $u=v$ and $\{u',v'\} \in E(\Lambda)$ or $u'=v'$ and $\{u,v\} \in E(\Gamma)$.

Both graphs and frames determine associated matrices, the entries of which may be over the real or complex field.  In the case of frames, we consider the finite-dimensional Hilbert space $\hil^d$, equipped with the standard inner product $\langle \ \cdot \ , \ \cdot \ \rangle$, where $\hil = \C$ or $\R$. In what follows, we denote by $I_n$, $0_n$, and $J_n$ the $n\times n$ identity matrix, zero matrix, and all-ones matrix, respectively.

For a graph $\Gamma$ on $n$ vertices, the $n \times n$ Laplacian matrix $L$ of $\Gamma$ is defined by
\[
L_{ij} = \bpar
 \deg(v_i)&& \textrm{ if } i=j\\
-1&& \textrm{ if } \{v_i,v_j\}\in E \\
0&&  \textrm{ otherwise.}
\eparr
\]
The Laplacian can also be written as $L(\Gamma)=D(\Gamma)-A(\Gamma)$, where $D(\Gamma)$ is the diagonal matrix with the degree of each vertex of $\Gamma$ along the diagonal and $A(\Gamma)$ is the \textit{adjacency} matrix with zeros along the diagonal and off-diagonal zero/one pattern corresponding to adjacencies between vertices.  More generally, to a graph $\Gamma$ we associate the following subset of $n \times n$ Hermitian (or real symmetric) matrices:
\begin{equation*}
    \CH(\Gamma) = \{ M \in \hil^{n\times n}: M=M^*, \ m_{ij} \neq 0 \iff \{v_i, v_j\} \in E(\Gamma) \text{ for }  i\neq j \}
\end{equation*}
where $^*$ denotes the conjugate transpose. All matrices in $\CH(\Gamma)$ must have real diagonal entries, that are otherwise unrestricted, and real eigenvalues. The Spectral Theorem highlights the crucial connection between a Hermitian matrix and its eigenvalues and an orthonormal basis of eigenvectors:

\begin{theorem}\label{spectral}
A Hermitian matrix $M$ may be diagonalized as $M=UDU^*$, where the eigenvalues of $M$ are the entries of the real diagonal matrix $D$ and $U$ is a unitary matrix whose columns are a complete set of corresponding unit eigenvectors of $M$.
\end{theorem}

The set of \textit{positive semidefinite} Hermitian matrices associated with a graph $\Gamma$ is defined by
\begin{equation*}
  \CH^+(\Gamma) = \{M\in \CH(\Gamma): \langle Mx, x \rangle \geq 0\textrm{ for all } x \in \hil^n \}.
\end{equation*}
Clearly any $M\in \CH^+(\Gamma)$ has nonnegative eigenvalues.  The {\it minimum semidefinite rank} of a graph is
\begin{equation*}
    \msr(\Gamma)=\min\{\rank(M):M\in \CH^+(\Gamma) \}.
\end{equation*}
Although the minimum semidefinite rank is not field-independent \cite{BBF10}, the named graphs that appear in this paper have equal minimum semidefinite rank over $\C$ and $\R$.

The Laplacian matrix of a graph of order $n$ is positive semidefinite with smallest eigenvalue $0$; note that the all-ones vector ${\bf 1}$ is an eigenvector of $L$ corresponding to $\lambda=0$.  The rank of $L$ is $n-k$, where $k$ is the number of connected components of the graph \cite{C97}.  Suppose $\Gamma$ is a connected graph on $n$ vertices.  Then $\rank(L(\Gamma))=n-1$ and $\msr(\Gamma)\leq n-1$.  Now suppose $\msr(\Gamma)=1$. Then there exists $M\in \CH^+(\Gamma)$ with only one nonzero eigenvalue $\lambda$, and  $M=u\lambda u^*$, where $u$ is its corresponding unit eigenvector. If any $u_{k}=0$, both the $k$th row and $k$th column of $M$ contain only zeros, which contradicts the assumption that $\Gamma$ is connected. Therefore, any connected graph $\Gamma$ with $\msr(\Gamma) = 1$ must be the complete graph $K_n$, and for $\Gamma \neq K_n$, $2\leq \msr(\Gamma) \leq n-1$.
In the remainder of this paper, we assume $\Gamma$ is connected since a matrix associated with a disconnected graph can be written as the direct sum of matrices for each component.

Denote by $q(\Gamma)$ the minimum number of distinct eigenvalues for any $M \in \CH(\Gamma)$. Clearly $1 \leq q(\Gamma) \leq n $. Now if $q(\Gamma)=1$, then $M = UDU^* = \lambda I_n$ for some $M \in \CH(\Gamma)$, so $\Gamma$ is the \textit{edgeless}, or totally disconnected, graph on $n$ vertices. Therefore, connected graphs must have $q(\Gamma) \geq 2$. Of particular interest are {\it dual multiplicity} graphs where $q(\Gamma)=2$. As a consequence of the last paragraph, $K_n$ is a dual multiplicity graph.

We now turn our attention to the frame representation of a graph.  A sequence of vectors $\CF=\{ f_i \} _{i=1}^n $ is a {\it finite frame} for a $d$-dimensional Hilbert space $ \hil^d $ if there exist constants 
$ 0 < A \leq B < \infty $ such that
\begin{equation} \label{framedef}
A\| x \|^2 \leq \sum_{i=1}^n |\langle x , f_i \rangle |^2 \leq B \|x \|^2 \qquad\mbox{ for all } x\in \hil^d.
\end{equation}
 The two constants $A$ and $B$ are called {\em frame bounds}.  The largest lower frame bound and smallest upper frame bound are called the {\em optimal} frame bounds.  If $A = B$, the frame is {\em tight} (or $B$-tight), and when $A=B=1$, the frame is a
 \textit{Parseval frame}.  Inequality (\ref{framedef}) is equivalent to $\mbox{span} \{ f_i \} _{i=1}^n = \hil^d$ (although this is not valid in infinite dimensions). The \textit{synthesis operator} $F:\hil^n \rightarrow \hil^d$ is given by
 $F(e_i)=f_i$, where $\{e_i\}_{i=1}^n$ is the canonical orthonormal basis for $\hil^n$. Its adjoint, the \textit{analysis operator} $F^*: \hil^d
 \rightarrow \hil^n$ is given by $F^*(x) = (\langle x, f_i \rangle)_{i=1}^n$. In what follows, we use the matrix representations of these operators, 
\[ F = \bmat | & & | \\ f_{1} & \cdots & f_{n}\\ | & & | \emat
\qquad \qquad \mbox{and} \qquad \qquad
F^*= 
\bmat
\textrm{---}\!\!\! &  f_1^* & \!\!\!\textrm{---} \\
 & \vdots &  \\
\textrm{---}\!\!\! & f_n^* & \!\!\!\textrm{---} \\
\emat
\]
for the $d\times n$ synthesis matrix and $n\times d$ analysis matrix, respectively.  The {\em frame operator} 
$S = FF^*:\hil^d \rightarrow\mathcal\hil^d $ is defined by $$FF^*(x)=\sum_{i=1}^n\langle x, f_i \rangle f_i = \sum_{i=1}^n f_if_i^*x,$$ and the
{\em Gram operator} (also {\em Gramian} or {\em Gram matrix}) $G = F^*F: \hil^n \rightarrow \hil^n $ is given by $g_{ij}=\langle f_j, f_i \rangle$ for $1\leq i,j \leq n$.  Note that $A\|x\|^2 \leq \langle Sx, x \rangle \leq B\|x\|^2$ for all $x\in \hil^d$, and $\CF$ is a Parseval frame if and only if $S=I_d$.  The matrix $S = FF^*$ is invertible, positive definite, and Hermitian.  Its eigenvalues are the nonzero eigenvalues of the positive semidefinite, Hermitian matrix $G = F^*F$.  The following result is well-known.

\begin{lemma} \label{gram}
The frame $\CF$ is a Parseval frame if and only if $G^2 = G$.  That is, a frame is Parseval if and only if its Gram matrix is an orthogonal projection.  
\end{lemma}

\begin{proof}
If $FF^* = I_d$, then $G^2 = G$.  If $G^2 = G$, then left-multiplying both sides by $F$ and right-multiplying both sides by $F^*$ gives $S^3 = S^2$.  Since $S$ is invertible, it follows that $S=I_d$.
\end{proof}
\vspace{0.35cm}

As investigated in \cite{AN18}, a frame $\CF=\{f_i\}_{i=1}^n$ for $\hil^d$ and a graph $\Gamma=(V,E)$ with $|V|=n$ may be {\it associated} with each other provided there is a one-to-one correspondence between vectors $f_i$ and vertices $v_i$ such that $\{v_i,v_j\}\in E$ if and only if $\langle f_i, f_j \rangle \neq 0$.  This is known as a {\em vector representation}, or for us a {\em frame representation}, of the graph.  Through an abuse of notation, we may refer to the vertices $v_i$ as the vectors $f_i$. 

We end this section by pointing out the meanings of $\msr(\Gamma)$ and $q(\Gamma)$ in the frame theoretic context and providing a specific example to illustrate the frame theoretic concepts encountered above.

Given a (connected) graph $\Gamma$ with $\msr(\Gamma)=k$, there exists some $M\in \CH^+(\Gamma)$ with rank $k$, which can be decomposed as $M = UDU^*$ by Theorem \ref{spectral}.  Letting $\widetilde D$ be the truncated matrix of only the nonzero eigenvalues of $M$ and $\widetilde U$ the corresponding truncated matrix of eigenvectors, we recognize $M$ as the Gram matrix of a frame: \[ M = \widetilde U \sqrt{\widetilde D}^* \sqrt{\widetilde D} \widetilde U^* = F^*F,\] where $F = \sqrt{\widetilde D} \widetilde U^*$ is a $k\times n$ matrix whose columns form a frame for $\hil^k$. Clearly the minimum semidefinite rank of a graph $\Gamma$ is the lowest dimension of a Hilbert space in which a frame associates with $\Gamma$.

If a graph is associated with a tight frame, then it is known as a {\em tight frame graph}.  Let $\CF = \{f_i\}_{i=1}^n$ be a frame for $\hil^d$ associated with a connected graph $\Gamma$.  The frame operator $S$ has eigenvalues $\lambda_1 \leq \lambda_2 \leq \hdots \leq \lambda_d$, where the optimal lower frame bound of $\CF$ is $A=\lambda_1$ and optimal upper frame bound is $B=\lambda_d$ \cite{CKP13}. If $\CF$ is tight, $A=B$, the Gramian $G$ must have only one nonzero eigenvalue, and $q(\Gamma) = 2$.  In fact, a connected graph is a tight frame graph if and only if it is a dual-multiplicity graph (\cite{AN18}, Theorem 5.2).

\begin{example} \label{diamondframe}
The diamond graph $K_4-e$ in Figure \ref{fig1} is a tight frame graph since
\[
F=
\frac{1}{\sqrt{10}}\bmat 
\frac{1}{\sqrt{2}} & \frac{-3}{\sqrt{2}} & 2 & 1\\
\frac{1}{\sqrt{2}} & \frac{3}{\sqrt{2}} & 1& 2\\
\emat
\]
is the synthesis matrix of a Parseval frame for $\R^2$ associated with the diamond graph.  Indeed, it is easy to check that $S = FF^* = I_2$ and $G = F^*F \in \CH^+(K_4-e).$  As mentioned in the introduction, a key advantage of frames over orthonormal bases is their redundancy, which makes frames robust to erasures.  For example, any $x\in\R^2$ can be reconstructed from the coefficients $\langle x, f_i \rangle$ as $x = Sx = \sum_{i=1}^4 \langle x, f_i \rangle f_i$, where the $f_i$ are the columns of $F$.  Suppose that, in the course of data transmission, $\langle x, f_1 \rangle$ and $\langle x, f_3 \rangle$ are lost.  Then $\widetilde \CF = \{f_2, f_4\}$ is still a spanning set for $\R^2$ and hence still a frame.  If $\widetilde S$ is the frame operator of $\widetilde \CF$, then $x$ can still be recovered as \[ x = \langle x, f_2 \rangle \widetilde S^{-1}(f_2) + \langle x, f_4 \rangle \widetilde S^{-1}(f_4). \]
We see in Figure \ref{spanning} that the original frame is robust to the loss of any two vectors. On the other hand, the sequence $\CF' = \{f_1, f_2, f_3, f_3\}$ is a frame for $\R^2$ that represents the diamond graph but which is robust to only one erasure, since reconstruction would be impossible in the event that both $\langle x, f_1 \rangle$ and $\langle x, f_2 \rangle$ were lost.

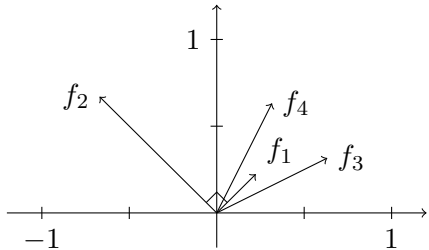
\begin{figure}[h]
\centering
\qquad
\begin{tikzpicture}[scale=2.3]
\draw [->] (0,0)-- (.223,.223)node[above right,scale=\spanscale]{$f_1$};
\draw [->] (0,0)--(-.67,.67)node[left,scale=\spanscale]{$f_2$};
\draw [->] (0,0)--(.632,.316)node[right,scale=\spanscale]{$f_3$};
\draw [->] (0,0)--(.316,.632)node[right,scale=\spanscale]{$f_4$};
\draw (-.06,.06)--(0,.12)--(.06,.06);
\draw[->] (-1.2,0) -- (1.2,0) coordinate (x axis);
\draw[->] (0,-.2) -- (0,1.2) coordinate (y axis);
\foreach \x/\xtext in {-1, -0.5/,0.5/, 1}
\draw (\x cm,1pt) -- (\x cm,-1pt) node[anchor=north,scale=\spanscale] {$\xtext$};
\foreach \y/\ytext in { 0.5/, 1}
\draw (1pt,\y cm) -- (-1pt,\y cm) node[anchor=east,scale=\spanscale] {$\ytext$};
\end{tikzpicture}
\vspace{-.5cm}
    \caption{The column vectors of the synthesis matrix in Example \ref{diamondframe}}
    \label{spanning}
\end{figure}

\end{example}

\section{Tight Frames}

In this section, we use a frame theoretic approach to give alternate proofs of some known results about tight frame graphs and to improve a result on embedding frame graphs in tight frame graphs.

\begin{lemma} \label{tightiffpars}
A graph $\Gamma$ is a tight frame graph if and only if it is a Parseval frame graph.
\end{lemma}

\begin{proof}
Suppose $\Gamma$ is a tight frame graph on $n$ vertices represented in $\hil^d$. Then there exists a tight frame $\CF=\{f_i\}_{i=1}^n \subseteq \hil^d$ with frame operator $S=FF^*=BI_d$ and Gramian $G = F^*F \in \CH^+(\Gamma)$. Since $\frac{1}{\sqrt{B}}F\frac{1}{\sqrt{B}}F^*=I_d$ and $\frac{1}{B}G$ has the same off-diagonal zero/nonzero pattern as $G$, $\frac{1}{\sqrt{B}}\CF$ is a Parseval frame representing $\Gamma$.  Conversely, if $\Gamma$ is a Parseval frame graph, then there exists a tight frame with frame bound $B=1$ representing $\Gamma$.
\end{proof}
\vspace{0.35cm}

The following simple but highly useful visual characterization of graphs that are not tight frame graphs was given as Corollary 3.4 of \cite{CGM14}, Corollary 4.5 of \cite{AAC13}, and Theorem 5.3 of \cite{AN18}:

\begin{proposition}\label{neighbor}
If $\Gamma$ is a tight frame graph, any two non-adjacent vertices of $\Gamma$ do not have exactly one common neighbor. 
\end{proposition}

\begin{proof}
By Lemma \ref{tightiffpars}, we may assume that $\Gamma$ is a Parseval frame graph, with corresponding Parseval frame $\CF$ and Gram matrix $G$.  Assume non-adjacent vertices $f_i$ and $f_j$ have a unique common neighbor $f_k$.  Then $\langle f_j, f_i \rangle = 0$, $\langle f_k, f_i \rangle \neq 0$, $\langle f_j, f_k \rangle \neq 0$, and $\langle f_l, f_i \rangle \langle f_j, f_l \rangle = 0$ for all $l\neq k$.  By Lemma \ref{gram}, 
\begin{equation*}
  0 = \langle f_j, f_i \rangle = \sum_{l=1}^n \langle f_l, f_i \rangle \langle f_j, f_l \rangle = \langle f_k, f_i \rangle \langle f_j, f_k \rangle,
\end{equation*}
a contradiction.
\end{proof}
\vspace{0.35cm}

As a consequence of the previous result, a connected tight frame graph with three or more vertices must be bridgeless.  In fact, we can say more.

\begin{corollary}
If $\Gamma$ is a connected tight frame graph with at least three vertices, then every edge of $\Gamma$ belongs to a 3-cycle or a 4-cycle.
\end{corollary}

\begin{proof}
Suppose $\{u,v\} \in E(\Gamma)$.  Since $\Gamma$ is connected with $|V(\Gamma)| \geq 3$, we may assume there exists $w\in V(\Gamma)$ such that $\{v,w\} \in E(\Gamma)$.  If $u$ is adjacent to $w$, then $\{u,v\}$ belongs to a 3-cycle.  Otherwise, if $u$ is not adjacent to $w$, then $\{u,v\}$ belongs to a 4-cycle by Proposition \ref{neighbor}.
\end{proof}
\vspace{0.35cm}

The redundancy (overcompleteness) of frames is the feature that makes them desirable in many applications when compared to orthonormal bases.  The following shows that redundancy is necessary for tight frames associated with connected (or merely not edgeless) graphs.  

\begin{proposition} \label{d<N}
Suppose $\CF = \{f_i\}_{i=1}^n$ is a tight frame for $\hil^d$ associated with a connected graph $\Gamma$.  If $d\geq 2$, then $d < n$.
\end{proposition} 

\begin{proof}
By Lemma \ref{tightiffpars}, assume $F$ is the synthesis matrix of a Parseval frame for $\hil^d$ with Gramian $G = F^*F \in \CH^+(\Gamma)$.  Since $S = FF^* = I_d$, the rows of $F$ are orthonormal.  If $d=n$, then $F$ is a unitary matrix, and $G = S = I_d$, which implies $\Gamma$ is edgeless, contradicting the assumption that $\Gamma$ is connected.
\end{proof}

\begin{corollary}
Suppose $\CF = \{f_i\}_{i=1}^n$ is a tight frame for $\hil^d$ associated with a connected graph $\Gamma$.  If $d\geq 2$, then zero is an eigenvalue of its Gramian $G$.
\end{corollary}

\begin{proof}
If $S$ is the frame operator of $\CF$, then $\rank(G) = \rank(S) = d < n$, by Proposition \ref{d<N}.
\end{proof}
\vspace{0.35cm}

Naimark's Theorem is a central result in the theory of Parseval frames \cite{CKP13}.  A {\it Naimark complement} of a Parseval frame $\CF=\{f_i\}_{i=1}^n$ for $\hil^d$ for $d<n$ can be viewed as a completion of the $d\times n$ synthesis matrix to an $n\times n$ unitary matrix \cite{CFM13}. As an explicit construction, begin with a Parseval frame $\CF=\{f_i\}_{i=1}^n$ for $\hil^d$ with Gramian $G=F^*F$. By Theorem \ref{spectral}, we can write $G=UDU^*$ where $U$ is unitary and $D=I_d \oplus 0_{n-d}$. Since
\begin{equation*}
\begin{aligned}
    I_n-G = I_n -UDU^* =U(I_n-D)U^*,
\end{aligned}
\end{equation*}
we see that $I_n-G$ is the Gram matrix of the Parseval frame $\widetilde\CF$ for $\hil^{n-d}$ whose synthesis matrix is the last $n-d$ rows of $U^*$.  And $I_n-G$ has the same off-diagonal zero/nonzero pattern as $G$, so $\CF$ and its complement $\widetilde\CF$ have the same associated graph (see also Theorem 2.1 of \cite{CGM14}). By Lemma \ref{tightiffpars}, we may refer to the {\em frame complement} in the case of a tight frame graph $\Gamma$ and note that $\msr(\Gamma)\leq \lfloor \frac{n}{2} \rfloor$.

\begin{proposition}\label{KnHn-1}
$K_n$ is the only connected graph on $n$ vertices represented by a tight frame for $\hil^{n-1}$.
\end{proposition}

\begin{proof}
If $\CF$ is a tight frame for $\hil^{n-1}$ with Gramian $G\in\CH^+(\Gamma)$, then the frame complement is a tight frame for $\hil$, but the only connected graph with $\msr(\Gamma)=1$ is $\Gamma = K_n$.
\end{proof}
\vspace{0.35cm}

Although it is proved in \cite{CGM14} that any graph may occur as an induced subgraph of a tight frame graph, with the same minimum semidefinite rank, the following frame theoretic approach often gives a smaller such embedding.

\begin{theorem}[\cite{CFM13}, Proposition 2.1]\label{tightcompletion}
Let $\CF=\{f_j\}_{j=1}^n$ be a frame for $\hil^d$ with synthesis matrix $F$. Suppose $S=FF^*$ has orthonormal eigenvectors $\{x_i\}_{i=1}^d$ corresponding to eigenvalues $A= \lambda_1 \leq \lambda_2 \leq \hdots \leq \lambda_k < \lambda_{k+1} = \hdots = \lambda_d = B$.  Then $\{f_j\}_{j=1}^n \cup \{h_i\}_{i=1}^k$ is a $B$-tight frame for $\hil^d$ where
\begin{equation*}
    h_i = (B-\lambda_i)^{\frac{1}{2}} x_i
\end{equation*}
for $1 \leq i \leq k$.
\end{theorem}

\begin{corollary} \label{min}
Any connected frame graph is an induced subgraph of a connected tight frame graph, represented by vectors in the same Hilbert space.  For a specific frame, the number of additional vectors required for the embedding is minimal.
\end{corollary}

\begin{proof}
Let $\Gamma$ be a connected graph represented by the frame $\CF=\{f_j\}_{j=1}^n \subseteq \hil^d$, and let $\{h_i\}_{i=1}^k$ be defined as in Theorem \ref{tightcompletion}.  Since $\CF$ is a spanning set for $\hil^d$ and each $h_i \neq 0$, it is impossible for $\langle h_i, f_j \rangle = 0$ for all $1\leq j \leq n$, which ensures that the corresponding tight frame graph, represented by $\{f_j\}_{j=1}^n \cup \{h_i\}_{i=1}^k$, is connected.  As noted in \cite{CFM13}, this construction adds a minimal number of vectors needed to complete $\CF$ to a tight frame.  Indeed, suppose the eigenvalues of $S = FF^*$ are $A= \lambda_1 \leq \lambda_2 \leq \hdots \leq \lambda_k < \lambda_{k+1} = \hdots = \lambda_d = B$ and there exist vectors $h'_i \in \hil^d$ for $1\leq i \leq k' < k$ such that  writing $H'$ for the matrix with columns $\{h'_i\}_{i=1}^{k'}$ yields $[F \ H'] [F \ H']^* = FF^* + H'H'^* = CI_d$ for some constant $C>0$.  Since $B = \lambda_d$ is the optimal upper frame bound of $\CF$, Inequality (\ref{framedef}) implies $B\leq C$. By a theorem of Weyl (\cite{HJ13}, Corollary 4.3.5), $C = \lambda_1(FF^* + H'H'^*) \leq \lambda_{1+k'}(FF^*) < B$, a contradiction.
\end{proof}

\begin{remark} \label{comparison}
For a given graph, we can contrast the the worst-case number of additional vectors required by the construction in Corollary \ref{min} with that of Theorem 4.2 of \cite{CGM14}.  By Corollary \ref{min}, the largest possible value for the additional number of vectors is $d-1$, in the case when all eigenvalues of $S$ are distinct.  On the other hand, the construction in the proof of Theorem 4.2 of \cite{CGM14} is a two-step process:  first, columns are adjoined to $F$ to yield a matrix with orthogonal rows; then more columns are adjoined to result in equal-norm rows, which may then be normalized.  The entire process could result in up to $\frac{(d-1)(d+2)}{2}$ additional vectors.
\end{remark}

We illustrate the difference between the two constructions compared in Remark \ref{comparison}.  If we take the frame in Example \ref{diamondframe} and delete a single vector to create the frame $\CF'$, the rows of the corresponding synthesis matrix $F'$ are no longer orthogonal, so the frame is no longer tight. Then $\CF'$ is a frame for $\hil^2$ representing a truncated version of Figure \ref{fig1} (either $P_3$ or $C_3$), and its frame operator $S$ has two distinct eigenvalues $A<B$.  (Although $C_3$ is a tight frame graph, both constructions may proceed for a non-tight frame representation of it.)  By Theorem \ref{tightcompletion}, only one additional vector is needed to complete $\CF'$ to a $B$-tight frame. That is to say the truncated graph occurs as an induced subgraph of some tight frame graph on four vertices. On the other hand, applying the construction method from Theorem 4.2 of \cite{CGM14} to $\CF'$, one vector is added to make the rows of $F'$ orthogonal and one additional vector must then be added to normalize the rows. The embedding is therefore not minimal for this frame. Note that the frame given in Example 6.10 of \cite{CGM14} corresponds to the same graph, and if either of the first two columns is deleted, then the construction method in Theorem 4.2 of \cite{CGM14} will only add a single vector; however, if either of the second two columns is deleted, then two additional vectors will be added.

For larger tight frame graphs, if a single frame vector is deleted (resulting in a non-tight frame), the embedding in Theorem 4.2 of \cite{CGM14} may add significantly more vectors to the frame than necessary to embed in a tight frame graph. Consider the tight frame graph $\Gamma=K_3 \cp K_2$ with tight frame given in Figure 2.1 of \cite{CGM14}. The deletion of any single vertex necessarily results in a non-tight frame graph.  Applying Theorem 4.2 of \cite{CGM14} adds between 3 and 5 vectors to the truncated frame while Theorem \ref{tightcompletion} appends only one vector. Figure \ref{embedcompare} below shows the tight frame graph created by deleting the first vector of the tight frame in Figure 2.1 of \cite{CGM14} and then applying the construction techniques in Theorem 4.2 of \cite{CGM14} and Theorem \ref{tightcompletion}, respectively. The highlighted vertices and edges constitute the truncated frame graph represented by the frame with synthesis matrix 
\[
F'=
\frac{1}{\sqrt{5}}\bmat
1 & 1 & -1 & 1 & 1\\
0 & 1 & 1 & -1 & 1\\
1 & 0 & 1 & 1 & -1\\
\emat
\]
and frame bounds $A=3/5$ and $B=1$;
the two algorithms construct Parseval frames with synthesis matrices
\[
F_1=
\frac{1}{3}\bmat
1 & 1 & -1 & 1 & 1&0&2&0\\
0 & 1 & 1 & -1 & 1&\sqrt{5}&0&0\\
1 & 0 & 1 & 1 & -1&\frac{1}{\sqrt{5}}&0&\frac{2\sqrt{30}}{5}\\
\emat
\ 
\mbox{and }
\
F_2=
\frac{1}{\sqrt{5}}\bmat
1 & 1 & -1 & 1 & 1& 0\\
0 & 1 & 1 & -1 & 1& 1\\
1 & 0 & 1 & 1 & -1& 1\\
\emat,
\] respectively.

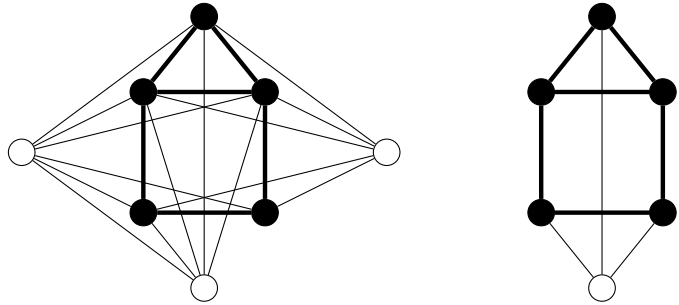
\begin{figure}[hbt]
\centering
    \begin{tikzpicture}
    \begin{scope}[every/.style={circle,draw,inner sep=0pt,minimum size=\mynmnodesz}]
        \node[fill=black,every] (v1) at (-.8,0) {};
        \node[fill=black,every] (v2) at (.8,0) {};
        \node[fill=black,every] (v3) at (0,2.6) {};
        \node[fill=black,every] (v4) at (-.8,1.6) {};
        \node[fill=black,every] (v5) at (.8,1.6) {};
        \node[every] (v6) at (-2.4,.8) {};
        \node[every] (v7) at (2.4,.8) {};
        \node[every] (v8) at (0,-1) {};
    \end{scope}
    \path (v1) edge[line width=1.65pt] node {} (v2);
    \path (v1) edge[line width=1.65pt] node {} (v4);
    \path (v2) edge[line width=1.65pt] node {} (v5);
    \path (v4) edge[line width=1.65pt] node {} (v5);
    \path (v4) edge[line width=1.65pt] node {} (v3);
    \path (v5) edge[line width=1.65pt] node {} (v3);
    \path (v6) edge node {} (v1);
    \path (v6) edge node {} (v2);
    \path (v6) edge node {} (v3);
    \path (v6) edge node {} (v4);
    \path (v6) edge node {} (v5);
    \path (v7) edge node {} (v1);
    \path (v7) edge node {} (v2);
    \path (v7) edge node {} (v3);
    \path (v7) edge node {} (v4);
    \path (v7) edge node {} (v5);
    \path (v8) edge node {} (v1);
    \path (v8) edge node {} (v4);
    \path (v8) edge node {} (v5);
    \path (v8) edge node {} (v6);
    \path (v8) edge node {} (v3);
    \end{tikzpicture}
\qquad
\qquad
\begin{tikzpicture}
    \begin{scope}[every/.style={circle,draw,inner sep=0pt,minimum size=\mynmnodesz}]
        \node[fill=black,every] (v1) at (-.8,0) {};
        \node[fill=black,every] (v2) at (.8,0) {};
        \node[fill=black,every] (v3) at (0,2.6) {};
        \node[fill=black,every] (v4) at (-.8,1.6) {};
        \node[fill=black,every] (v5) at (.8,1.6) {};
        \node[every] (v6) at (0,-1) {};
    \end{scope}
    \path (v1) edge[line width=1.65pt] node {} (v2);
    \path (v1) edge[line width=1.65pt] node {} (v4);
    \path (v2) edge[line width=1.65pt] node {} (v5);
    \path (v4) edge[line width=1.65pt] node {} (v5);
    \path (v4) edge[line width=1.65pt] node {} (v3);
    \path (v5) edge[line width=1.65pt] node {} (v3);
    \path (v6) edge node {} (v1);
    \path (v6) edge node {} (v2);
    \path (v6) edge node {} (v3);
    \end{tikzpicture}
    \caption{The graphs represented by tight frames with synthesis matrices $F_1$ (obtained via Theorem 4.2 of \cite{CGM14}) and $F_2$ (obtained via Theorem \ref{tightcompletion}), respectively}
    \label{embedcompare}
  \end{figure}
  
\begin{remark}
If $F'$ does not have a scaling factor of  $\frac{1}{\sqrt{5}}$, the construction technique from Theorem 4.2 of \cite{CGM14} appends only a single vector to make the rows orthogonal and equal norm. But even without scaling, deleting any of the last three vectors from the original frame in Figure 2.1 of \cite{CGM14} results in three vectors being appended to achieve orthogonality, using the technique of \cite{CGM14}.
\end{remark}

\section{Line Graphs}

It is natural to try to construct frames for graphs using the Laplacian matrix since the resultant frame has a graph theoretic interpretation. This approach, however, requires us to consider the Laplacian of a graph and the vector representation of its line graph.

Given a graph $\Gamma$, the {\it (unoriented) incidence matrix} $B^u(\Gamma)$ is the $|V|\times |E|$ matrix with $b^u_{ij}=1$ if vertex $v_i$ is an endpoint of the $j$-th edge and $0$ otherwise.  We may assign an {\em orientation} to $\Gamma$ by associating each edge $\{v_i,v_k\}$ arbitrarily with one of the ordered pairs $(v_k,v_i)$ (in which case, we say the edge is negatively incident to $v_k$ and positively incident to $v_i$) or $(v_i,v_k)$; the {\it oriented incidence matrix} $B(\Gamma)$ has $b_{ij}=1=-b_{kj}$ if the $j$-th edge $\{v_i, v_k\}$ is associated to $(v_k, v_i)$.

For a connected graph $\Gamma$ with $n$ vertices, the Laplacian matrix $L = L(\Gamma)$ is a real, positive semidefinite matrix of rank $n-1$.  The eigenvalue $\lambda_0 = 0$ corresponds to the one-dimensional eigenspace spanned by the all-ones eigenvector $x_0 = \bf{1}$.  Now, $L(\Gamma)$ can be decomposed as $L(\Gamma) = B(\Gamma)B(\Gamma)^*$, where $B = B(\Gamma)$ is an oriented incidence matrix of $\Gamma$.  Since the frame operator must be full rank, the authors of \cite{BBCO17} create an $(n-1)\times (n-1)$ matrix $L_0$ by restricting $B$ to the $(n-1)$-dimensional subspace spanned by the orthonormal eigenvectors $x_1, \ldots, x_{n-1}$ of $L$ excluding $x_0$.  They let $X = [ x_1 \ x_2 \ \cdots \ x_{n-1} ]$ and consider $F = X^* B$ and $L_0 = FF^*$.  Then $L_0$ is the (full rank) frame operator of the frame for $\hil^{n-1}$ with synthesis matrix $F$, whose spectrum consists of the nonzero eigenvalues of $L$ (\cite{BBCO17}, Lemma 5).  We will refer to a frame constructed in this manner as being constructed by the {\em Laplacian method}.  

The only connected graph which yields a tight frame constructed by the Laplacian method is $K_n$, by Proposition \ref{KnHn-1}.  However, this frame will not represent $K_n$ in the sense used in this paper; that is, the Gram matrix of the frame will not belong to $\CH^+(K_n)$.  To establish the relationship between the frame and the graph, we instead are motivated by the well-known connection between the unoriented incidence matrix of a graph and the adjacency matrix of its line graph.  Given a graph $\Rho$ (pronounced ``Rho"), called the {\it root graph}, its {\it line graph} $\Gamma=\SL(\Rho)$ is the graph created by assigning a vertex in $V(\Gamma)$ to each edge in $E(\Rho)$, so $|V(\Gamma)|=|E(\Rho)|$, and calling two vertices in $\Gamma$ adjacent if and only if their corresponding edges are incident in $\Rho$.  Every connected line graph, with the exception of $K_3$, determines a unique connected root graph (up to isomorphism) \cite{W32}.

\begin{proposition} \label{line-root}
Let $\Rho$ be a graph with oriented incidence matrix $B=B(\Rho)$, and let $\Gamma = \SL(\Rho)$.  Then $BB^* = L(\Rho)$ and $B^*B \in \CH^+(\Gamma)$.  
\end{proposition}

\begin{proof}
As noted above, for any graph with oriented incidence matrix $B$, the Laplacian $L=BB^*$.  Since any off-diagonal entry of $B^*B$ is $1$ or $-1$ if two edges share an endpoint and $0$ otherwise, $B^*B \in \CH^+(\Gamma)$, by the definition of the line graph. (Note that for the unoriented incidence matrix $B^u(\Rho)$, $B^{u*}B^u-2I$ is the adjacency matrix of $\Gamma$.)
\end{proof}
\vspace{0.35cm}

The previous result allows us to view a given line graph $\Gamma$ of order $m$ in terms of its root graph $\Rho$ of order $n\leq m + 1$ (and for a root graph with many edges, $n$ can be significantly less than $m$) and build a frame corresponding to $\Gamma$ with rank $n-1$.

\begin{lemma} \label{samegraph}
Let $\Gamma$ be a connected line graph with $m$ vertices, and let $B = B(\Rho)$ be an $n\times m$ oriented incidence matrix of its root graph $\Rho$.  Suppose $F = X^*B$ is a frame for $\hil^{n-1}$ constructed by the Laplacian method applied to the graph $\Rho$.  Then $F^*F \in \CH^+(\Gamma)$.
\end{lemma}

\begin{proof}
It is easy to see that the orthogonal projection $XX^*$ of $\hil^n$ onto the eigenspace spanned by $\{x_1, x_2, \ldots, x_{n-1}\}$ satisfies $XX^* f = f - \frac{1}{n} \langle f, x_0 \rangle x_0$ for any $f\in \hil^n$ where $x_0 = {\bf 1}$.  For any columns $b_i, b_j$ of $B$,
\begin{equation*}
\begin{aligned}
\langle X^* b_i, X^* b_j \rangle = \langle b_i, XX^* b_j \rangle 
= \langle b_i, b_j - \tfrac{1}{n} \langle b_j, x_0 \rangle x_0 \rangle  
=\langle b_i, b_j \rangle.
\end{aligned}
\end{equation*}
By Proposition \ref{line-root}, $B^*B \in \CH^+(\Gamma)$, so $F^*F \in \CH^+(\Gamma)$.  
\end{proof}
\vspace{0.35cm}

The Laplacian method allows us to demonstrate the tightness of a new family of graphs, the line graphs of complete graphs.  We know that $\msr(\SL(K_n)) = n-2$ by \cite{P12}.  The Laplacian method generates a tight frame for $\hil^{n-1}$, and a modification of this frame allows us to obtain a tight frame for $\hil^{n-2}$, corresponding to $\SL(K_n)$.

\begin{theorem}\label{LKn}
For $n\geq 3$, if $\Rho = K_n$, then $\Gamma = \SL(\Rho)$ is a tight frame graph for $\hil^{n-1}$ and $\hil^{n-2}$.
\end{theorem}

\begin{proof}
Let $\Rho = K_n$ and $\Gamma = \SL(\Rho)$.  Apply the Laplacian method to construct the frame $F = X^*B$ for $\hil^{n-1}$, where $B = B(\Rho)$ is an oriented incidence matrix of $\Rho$. By Lemma \ref{samegraph}, $F^*F \in \CH^+(\Gamma)$.  Since $\lambda_1 = \lambda_2 = \cdots = \lambda_{n-1} = n$ are the nonzero eigenvalues of $L = BB^*$ (see, for example, \cite{C97}),
\begin{equation*}
L_0 = FF^*
    = X^*BB^*X
    = X^*L X
    = nI_{n-1}.
\end{equation*}
Thus $F$ is a tight frame associated to $\Gamma$ in dimension $n-1$.

To show that $\Gamma = \SL(K_n)$ is in fact a tight frame graph in dimension $n-2$, we use the construction from \cite{AIM08} in the proof that the minimum rank of $\Gamma$ is $n-2$ (also employed in \cite{P12} to show $\msr(\Gamma) = n-2$).  Let $D$ denote an oriented incidence matrix of $K_{n-1}$, and let $C = I_{n-1} - \frac{1}{n-1} J_{n-1}$. Consider the $(n-1)\times \frac{n(n-1)}{2}$ matrix $M = [C \ D]$.  Easy computations show that $C^*C = C^2 = C$ and $C^*D = D$.  The matrix 
\begin{equation*}
M^*M = \left[ \begin{array}{c c} C^*C & C^*D\\ D^*C & D^*D \end{array} \right]  =  \left[ \begin{array}{c c} C & D\\ D^* & D^*D \end{array} \right]
\end{equation*}
has nonzero entries in the top left corner of the partition, corresponding to all edges between each vertex of $K_{n-1}$ and the $n$th vertex of $K_n$, which are all incident with one another; has nonzero entries in the top right (and bottom left) corner of the partition, corresponding to all edges of $K_{n-1}$ incident with the edges connecting the vertices of $K_{n-1}$ to the $n$th vertex of $K_n$; and has nonzero entries in the bottom right corner of the partition, corresponding to all edge-incidences of $K_{n-1}$ (by Proposition \ref{line-root}). That is, $M^*M \in \CH^+(\Gamma)$.

The eigenvector $x_0 = {\bf 1}$ of $L(K_{n-1})$ is also an eigenvector of $C$, corresponding to the same eigenvalue $\lambda_0 = 0$.  For $1\leq j\leq {n-2}$, set $x_j(j) = -1$, $x_j(n-1) = 1$, and $x_j(i) = 0$ for $i\neq j, n-1$; the eigenvectors $x_1, x_2, \ldots, x_{n-2}$ of $L(K_{n-1})$ are also eigenvectors of $C$, corresponding to $\lambda_j = n-1$ and $\lambda_j = 1$ for all $j$, respectively.  Since
\[ MM^* = CC^* + DD^* = C + L(K_{n-1}), \]
$\lambda = 0$ is an eigenvalue of multiplicity 1 of $MM^*$ and $\lambda = n$ is an eigenvalue of multiplicity $n-2$.  Let $\{\widetilde x_0, \widetilde x_1, \widetilde x_2, \ldots \widetilde x_{n-2}\}$ be a corresponding orthonormal set of eigenvectors of $MM^*$.  The Laplacian method now can proceed exactly as it would for a Laplacian matrix:  let $\widetilde X = [\widetilde x_1 \ \widetilde x_2 \ \cdots \ \widetilde x_{n-2}]$ and $\widetilde F = \widetilde X^*M$.  This time,
\begin{equation*}
\widetilde F \widetilde F^* = \widetilde X^*MM^*\widetilde X = nI_{n-2},
\end{equation*}
and $\widetilde F$ is a tight frame for $\hil^{n-2}$.  As in the proof of Lemma \ref{samegraph}, $\widetilde X \widetilde X^* f = f - \langle f, \widetilde x_0 \rangle \widetilde x_0$ for any $f\in \hil^{n-1}$.  If $m_j$ is a column of $M$, then it is a column of $C$ or a column of $D$.  In either case, it is easy to see that $\langle m_j, \widetilde x_0 \rangle = 0$.  Therefore, for any columns $m_i, m_j$ of $M$,
\begin{equation*}
\begin{aligned}
\langle \widetilde X^* m_i, \widetilde X^* m_j \rangle = \langle m_i, \widetilde X \widetilde X^* m_j \rangle 
= \langle m_i, m_j-\langle m_j, \widetilde x_0 \rangle \widetilde x_0 \rangle  
=\langle m_i, m_j \rangle,
\end{aligned}
\end{equation*}
and $M^*M \in \CH^+(\Gamma)$ implies $\widetilde F^*\widetilde F \in \CH^+(\Gamma)$.
\end{proof}
\vspace{0.35cm}

By Lemma 5.1 of \cite{AFM19}, if a graph on $n$ vertices is a tight frame graph for $\hil^2$, then it is a tight frame graph for $\hil^d$ for all $2\leq d \leq \lfloor \frac{n}{2} \rfloor$.  It remains an open question (see \cite{AFM19}) whether the implication holds for graphs $\Gamma$ with $\msr(\Gamma) > 2$.  In particular, since $|\SL(K_n)| = \frac{1}{2}n(n-1)$, it remains to be determined whether $\SL(K_n)$ is a tight frame graph for $\hil^d$ whenever $n \leq d \leq \lfloor \frac{1}{4}n(n-1) \rfloor$.  Although it is well known that the complete graph on $n$ vertices is a tight frame graph in all possible dimensions, we provide an algorithmic proof here that utilizes a slight variant of the Laplacian method.

\begin{theorem}\label{Sn1Kn}
For $n\geq 2$, if $\Rho = S_{n+1}$, then $\Gamma = \SL(\Rho) = K_n$ is a tight frame graph for $\hil^d$ for $1\leq d \leq n-1$.
\end{theorem}

\begin{proof}
Let $\Rho=S_{n+1}$, $\Gamma=\SL(\Rho)=K_n$, and $B$ be an $(n+1) \times n$ oriented incidence matrix of $\Rho$; then $L = BB^*$ is the Laplacian matrix of $\Rho$, and $B^*B \in \CH^+(\Gamma)$, by Proposition \ref{line-root}. To illustrate,
\begin{equation*}
B= 
\bmat
-1 & -1 & \hdots & \hdots & -1 \\
1 & 0 & 0 & \hdots & 0 \\
0& 1&0&0&0\\
0&0&\ddots&0&0\\
\vdots&0&0&\ddots&0\\
0&0&0&0&1\\
\emat
\qquad
\mbox{ and } 
\qquad 
L = 
\bmat
n & -1 & -1 & \hdots & \hdots &-1\\
-1 & 1 & 0 & 0 & 0 & 0 \\
-1&0&1&0&0&0\\
\vdots&0&0&\ddots&0&0\\
\vdots&0&0&0&\ddots&0\\
-1&0&0&0&0&1\\
\emat.
\end{equation*}
Recall that $L$ has eigenvalue $\lambda_0=0$ corresponding to the eigenvector $x_0=\textbf{1}$. It is easy to see that $\lambda_n=n+1$ is an eigenvalue with eigenvector $x_n=(n,-1 ,\hdots  ,-1)^*$.  The remaining eigenvalues are $\lambda_1=\lambda_2=\hdots=\lambda_{n-1}=1$ since the columns of the $(n+1) \times (n-1)$ matrix 
\begin{equation*}
\widetilde X=
\bmat
0 & 0 & \hdots &\hdots&\hdots & 0 \\
\sqrt{\frac{n-1}{n}} & 0 & 0 & \hdots &\hdots& 0 \\
\frac{-\sqrt{\frac{n-1}{n}}}{n-1} &\sqrt{\frac{n-2}{n-1}} & \ddots &\ddots &\hdots&0\\
\frac{-\sqrt{\frac{n-1}{n}}}{n-1} &\frac{-\sqrt{\frac{n-2}{n-1}}}{n-2} &\sqrt{\frac{n-3}{n-2}} & \ddots &0&0\\
\vdots&\frac{-\sqrt{\frac{n-2}{n-1}}}{n-2} & \frac{-\sqrt{\frac{n-3}{n-2}}}{n-3} & \ddots &0& 0\\
\vdots&\vdots& \frac{-\sqrt{\frac{n-3}{n-2}}}{n-3}&\ddots &\sqrt{\frac{2}{3}} & 0 \\
\vdots&\vdots&\vdots&\ddots&\frac{-\sqrt{\frac{2}{3}}}{2}&\sqrt{\frac{1}{2}}\\
\frac{-\sqrt{\frac{n-1}{n}}}{n-1} & \frac{-\sqrt{\frac{n-2}{n-1}}}{n-2} & \frac{-\sqrt{\frac{n-3}{n-2}}}{n-3} &\hdots & \frac{-\sqrt{\frac{2}{3}}}{2}&-\sqrt{\frac{1}{2}}
\emat
\end{equation*}
form an orthonormal eigenbasis corresponding to eigenvalue $\lambda = 1$.
\newline
\indent
Let $\widetilde{X}_d$ be the matrix consisting of $d$ columns of $\widetilde X$, with any $n-1-d$ columns of $\widetilde X$ deleted except for the first. Then $F_d=\widetilde{X}^*_dB$ is the synthesis matrix of a Parseval frame for $\mathbb{H}^d$ since its frame operator is the identity:
\begin{equation*}
\begin{aligned}
F_dF_d^*  =\widetilde{X}_d^*B B^*\widetilde{X}_d
=\widetilde{X}_d^*L\widetilde{X}_d
=I_{d}.\\
\end{aligned}
\end{equation*} 

To show that the Gramian $F_d^*F_d \in \CH^+(\Gamma)$, we claim that $\langle r_i, r_j \rangle \neq 0$ for any distinct rows $r_i, r_j$ of $\widetilde X_d$ with $2\leq i, j \leq n+1$; we may assume $i<j$. Let $J\subset \{1,2,\hdots, n-1\}$ be the subset of indices for the columns of $\widetilde{X}_d$ and let $J'=J\cap \{1,\hdots,i-1\}$.  If $i-1 \notin J'$, then $i\geq 3$ and 
\begin{equation*}
\langle r_i, r_j \rangle = \sum_{k\in J'} \frac{1}{(n-k)(n-k+1)} > 0.
\end{equation*}
If $i-1 \in J'$, then, allowing the first sum in the expressions below to be empty when $i=2$, 
\begin{equation*}
\langle r_i, r_j \rangle = \sum_{\substack{k\in J'\\ k\neq i-1}} \frac{1}{(n-k)(n-k+1)} - \frac{1}{n-i+2} \leq \sum_{k=1}^{i-2} \frac{1}{(n-k)(n-k+1)} - \frac{1}{n-i+2} < 0,
\end{equation*}
where the last inequality follows from a straightforward inductive argument for $i\geq 3$.
So for any columns $f_i, f_j$ of $F_d$ with $1\leq i,j \leq n$ and $i\neq j$,

\begin{equation*}
\begin{aligned}
\langle f_i, f_j \rangle & = \langle \widetilde{X}_d^*b_i,\widetilde{X}_d^*b_j \rangle 
  =  \langle - r_1 + r_{i+1},- r_1 + r_{j+1} \rangle 
  =  \langle r_{i+1}, r_{j+1} \rangle \neq 0.
\end{aligned}
\end{equation*}
\end{proof} 

\begin{remark}
If the last $n-d-1$ columns of $\widetilde X$ were deleted in the proof of Theorem \ref{Sn1Kn}, then the resulting frame would have its last vector repeated $n-d$ times, so it is better to delete alternating columns and have fewer identical vectors in the frame.
\end{remark}

Although both line graphs in Theorems \ref{LKn} and \ref{Sn1Kn} are tight frame graphs, only one of the root graphs ($K_n$) is.  The question of whether or not a line graph is a tight frame graph is independent of whether or not its root graph is a tight frame graph.  We illustrate this with the next example.

\begin{example} \label{tightnottight}
The table below is illustrative, not comprehensive.  The reference column refers to the numbering in this paper or a citation.  Let $\Gamma = \SL(\Rho)$.

\begin{center}
\renewcommand{\arraystretch}{1.5}
\begin{tabular}{|c|ccc|ccc|}
\hline
& \multicolumn{3}{c|}{$\Gamma$ tight} & \multicolumn{3}{c|}{$\Gamma$ not tight}\\
\hline
\multirow{3}{*}{$\Rho$ tight}  & $\Rho$ & $\Gamma$ & Ref. & $\Rho$ & $\Gamma$ & Ref.\\
\cline{2-7}
&$K_n$ & $\SL(K_n)$ & \ref{Sn1Kn}, \ref{LKn} & $K_2 \Box K_3$ & \!\!$\SL(K_2 \Box K_3)$ &\ref{K_{2,n}}, \ref{indP_4}\\
&$C_3$, $C_4$&$C_3$, $C_4$ & \cite{AAC13} &&&\\
\hline
\multirow{3}{*}{$\Rho$ not tight}&$S_{n+1}$ ($n\geq 3$) & $K_n$& \ref{neighbor}, \ref{Sn1Kn} & $C_n$ ($n\geq5$)  & $C_n$& \ref{neighbor}\\
& $O_n$ ($n\geq4$) & $\SL(O_n)$& \ref{neighbor}, \ref{dup} & $P_n$ ($n\geq 4$) & $P_{n-1}$ & \ref{neighbor}\\
& $K_{2,n}$ ($n\geq3$) & \ $K_2 \Box K_n$ \ & \ref{K_{2,n}} & &&\\
\hline
\end{tabular}
\end{center}

\end{example}

\vspace{0.35cm}

The necessary condition from Proposition \ref{neighbor} for a frame graph to be tight can be converted to a necessary condition on the root graph of a tight frame line graph.  We state this precisely in the next result, used to construct a tight frame root graph with a line graph that is not a tight frame graph in Example \ref{tightnottight}.

\begin{lemma}\label{indP_4}
If $\Gamma =\SL(\Rho)$ and $\Rho$ contains an induced path on $4$ vertices, then $\Gamma$ is not a tight frame graph.
\end{lemma}

\begin{proof}
If $P_4$ is an induced subgraph of $\Rho$, then $\Gamma$ has $P_3$ as an induced subgraph, such that this $P_3$ is a subgraph of no $C_4$ in $\Gamma$.  So there are two non-adjacent vertices in $V(\Gamma)$ that have a unique common neighbor and, by Proposition \ref{neighbor}, $\Gamma$ cannot be a tight frame graph.
\end{proof}
\vspace{0.35cm}

Conversely, certain easily observed features of the line graph serve as necessary conditions for the root graph to be associated with a tight frame.  In a connected graph, let {\em pendant vertex} denote a vertex of degree 1, and let {\em pendant triangle} denote a $K_3$ subgraph such that exactly two of the three vertices have degree 2. 

\begin{observation}
If $\Gamma=\SL(\Rho)$ has a pendant vertex, then $\Rho$  is not a tight frame graph.
\end{observation}

\begin{proof}
If there is a degree-one vertex in $\Gamma$, then there is an edge $\{u, v\} \in E(\Rho)$ that is incident with exactly one other edge $\{v, w\}$ in $\Rho$.  Then $v$ is the unique neighbor of the non-adjacent vertices $u$ and $w$, and $\Rho $ is not a tight frame graph, by Proposition \ref{neighbor}. 
\end{proof}

\begin{observation}
If $\Gamma=\SL(\Rho)$ has a pendant triangle, then $\Rho$ is not a tight frame graph.
\end{observation}

\begin{proof}
The root graph of $K_3$ may be either $K_3$ or $K_{1,3}$.  However, if $\Gamma$ has a pendant triangle, then $\Rho$ must contain a claw ($K_{1,3}$) as an induced subgraph.  The two degree-two vertices of the pendant triangle in $\Gamma$ must correspond to the two edges of the claw that are not incident to any other edges in $\Rho$.  It follows that each of these two edges has an endpoint that is a degree-one vertex; that is, $\Rho$ contains two non-adjacent vertices with a unique neighbor (the central vertex of the claw).  By Proposition \ref{neighbor}, $\Rho$ is not a tight frame graph.
\end{proof}
\vspace{0.35cm}

When the root graph has the same number of vertices as, or one more than, its corresponding line graph, that is a tight frame graph, we can give a complete characterization of the root graph.  We first cite two preliminary results.

\begin{lemma} [\cite{CGM14}, Theorem 6.3] \label{Chen_6.3} 
Let $\Gamma$ be a tight frame graph on $n$ vertices.  The graph created by duplicating one vertex in $\Gamma$ (that is, introducing a new vertex adjacent to one vertex $u$ in $\Gamma$ and to all the neighbors of $u$) is a tight frame graph.
\end{lemma}

\begin{lemma} [\cite{AAC13}, Proposition 4.7] \label{Ahmadi_4.7}
The graph on $n\geq4$ vertices created by deleting a single edge of $K_n$ is a tight frame graph.
\end{lemma}

We only utilize the previous result in the case $n=4$.  The proof of Proposition 4.7 of \cite{AAC13} for that case contains an error, resulting in the construction of $C_4$ instead of $K_4$ minus an edge.  However, the result remains true, by Example \ref{diamondframe}.  Lemmas \ref{Chen_6.3} and \ref{Ahmadi_4.7} imply the following:

\begin{corollary}\label{dup}
Each graph in the sequence in Figure \ref{dupfig}, obtained by iteratively applying Lemma \ref{Chen_6.3} to the tight frame graph of Example \ref{diamondframe}, is a tight frame graph.
\end{corollary}
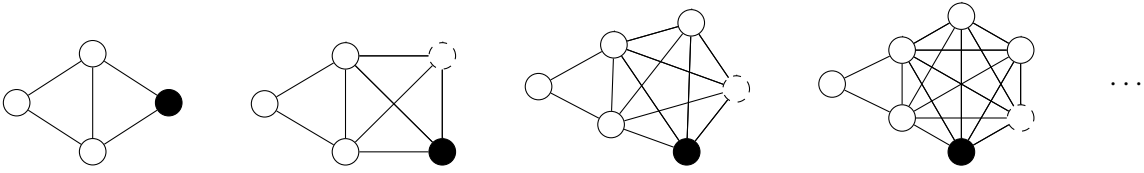
\begin{figure}[hbt]
\centering
    \begin{tikzpicture}
    \begin{scope}[every node/.style={circle,draw,inner sep=0pt,minimum size=\mynmnodesz}]
        \node (v1) at (1,1.3) {};
        \node (v2) at (0,.65) {};
        \node[fill=black] (v3) at (2,.65) {};
        \node (v4) at (1,0) {};
    \end{scope}
    \path (v1) edge node {} (v2);
    \path (v1) edge node {} (v3);
    \path (v1) edge node {} (v4);
    \path (v2) edge node {} (v4);
    \path (v3) edge node {} (v4);
    \end{tikzpicture}
    \qquad
    \begin{tikzpicture}
    \begin{scope}[every node/.style={circle,draw,inner sep=0pt,minimum size=\mynmnodesz}]
        \node (v1) at (-1.7,0) {};
        \foreach \x in {2,...,5}
        {
    \pgfmathparse{45+\x*360/4}
    \ifthenelse{\equal{\x}{3}}{\node[draw,circle,inner sep=0, fill=black] (v\x) at (\pgfmathresult:.9cm) {};}
    {\ifthenelse{\equal{\x}{4}}{\node[draw,circle,inner sep=0, style=dashed] (v\x) at (\pgfmathresult:.9cm) {};}
    {\node[draw,circle,inner sep=0] (v\x) at (\pgfmathresult:.9cm) {};}}
    }
    \end{scope}
    \path (v1) edge node {} (v5);
    \path (v1) edge node {} (v2);
    \foreach \x in {2,...,5}{
    \foreach \y in {3,...,5}{
    \path (v\x) edge node {} (v\y);
   }}
    \end{tikzpicture}
    \qquad
    \begin{tikzpicture}
    \begin{scope}[every node/.style={circle,draw,inner sep=0pt,minimum size=\mynmnodesz}]
         \node (v1) at (-1.7,0) {};
        \foreach \x in {2,...,6}
        {
    \pgfmathparse{70+\x*360/5}
    \ifthenelse{\equal{\x}{3}}{\node[draw,circle,inner sep=0, fill=black] (v\x) at (\pgfmathresult:.9cm) {};}
    {\ifthenelse{\equal{\x}{4}}{\node[draw,circle,inner sep=0, style=dashed] (v\x) at (\pgfmathresult:.9cm) {};}
    {\node[draw,circle,inner sep=0] (v\x) at (\pgfmathresult:.9cm) {};}}
    }
    \end{scope}
    \path (v1) edge node {} (v6);
    \path (v1) edge node {} (v2);
    \foreach \x in {2,...,6}{
    \foreach \y in {3,...,6}{
    \path (v\x) edge node {} (v\y);
   }}
    \end{tikzpicture}
    \qquad
    \begin{tikzpicture}
    \begin{scope}[every node/.style={circle,draw,inner sep=0pt,minimum size=\mynmnodesz}]
        \node (v1) at (-1.7,0) {};
        \foreach \x in {2,...,7}
        {
    \pgfmathparse{90+\x*360/6}
    \ifthenelse{\equal{\x}{3}}{\node[draw,circle,inner sep=0, fill=black] (v\x) at (\pgfmathresult:.9cm) {};}
    {\ifthenelse{\equal{\x}{4}}{\node[draw,circle,inner sep=0, style=dashed] (v\x) at (\pgfmathresult:.9cm) {};}
    {\node[draw,circle,inner sep=0] (v\x) at (\pgfmathresult:.9cm) {};}}
    }
    \end{scope}
    \path (v1) edge node {} (v7);
    \path (v1) edge node {} (v2);
    \foreach \x in {2,...,7}{
    \foreach \y in {3,...,7}{
    \path (v\x) edge node {} (v\y);
   }}
    \node[draw=none] (v8) at (2.2,0) {$\cdots$};
    \end{tikzpicture}
    \caption{Duplicating a vertex}
    \label{dupfig}
\end{figure}

Let $O_n$ be the graph on $n\geq3$ vertices constructed by taking the star graph $S_n$ and connecting two of its degree-one vertices with an edge (note $O_3=C_3=K_3$). Then $\SL(O_n)$ is a tight frame graph by Corollary \ref{dup}.
\begin{figure}[H] 
\centering
    \begin{tikzpicture}[rotate=180]
    \begin{scope}[every node/.style={circle,draw,inner sep=0pt,minimum size=\mynmnodesz}]
        \node (v1) at (2,.3) {};
        \node (v2) at (2,1.7){};
        \node(v3) at (1,1){};
        \node(v4) at (.2,.3){};
        \node(v5) at (0,1){};
        \node(v6) at (.2,1.7){};
    \end{scope}
    \path (v1) edge node {} (v2);
    \path (v1) edge node {} (v3);
    \path (v2) edge node {} (v3);
    \path (v3) edge node {} (v4);
    \path (v3) edge node {} (v5);
    \path (v3) edge node {} (v6);
    \end{tikzpicture}
    \qquad\qquad
    \begin{tikzpicture}
    \begin{scope}[every node/.style={circle,draw,inner sep=0pt,minimum size=\mynmnodesz}]
        \node (v1) at (-1.7,0) {};
        \foreach \x in {2,...,6}
        {
    \pgfmathparse{70+\x*360/5}
      {\node[draw,circle,inner sep=0] (v\x) at (\pgfmathresult:.9cm) {};}
    }
    \end{scope}
    \path (v1) edge node {} (v6);
    \path (v1) edge node {} (v2);
    \foreach \x in {2,...,6}{
    \foreach \y in {3,...,6}{
    \path (v\x) edge node {} (v\y);
   }}
    \end{tikzpicture}
    \caption{$O_6$ and $\SL(O_6)$}
    \label{O_6}
  \end{figure} 

\begin{proposition}\label{Pnn1V}
Suppose $\Gamma=\SL(\Rho)$ is a tight frame graph on $n\geq 3$ vertices with connected root graph $\Rho$.
\begin{enumerate}
    \item[(a)] If $|V(\Rho)| = n$, then $\Rho \in \{C_3, C_4, O_n\}$.
    \item[(b)] If $|V(\Rho)| = n+1$, then $\Rho = S_{n+1}$.
\end{enumerate}
\end{proposition}

\begin{proof}
(a) \ Suppose $\Rho \notin \{C_3,C_4,O_n \}$ has $|E(\Rho)|=n\geq 3$ and $|V(\Rho)| = n$.  Then $\Rho$ has a single, induced cycle on $k\geq3$ vertices.  We show that $\Rho$ must contain an induced path on four vertices, from which the result follows by Lemma \ref{indP_4}.  

Case 1: If $k=3$, since $\Rho \notin \{ O_n,C_3\}$, $n\geq 5$, and there must be two non-adjacent vertices adjacent to vertices on the induced cycle, such that these connecting edges are not incident with one another, or two vertices adjacent to each other, such that only one is adjacent to a vertex on the induced cycle.  In either situation, $\Rho$ contains an induced path on four vertices.

Case 2: If $k=4$, since $\Rho \neq C_4$, $n\geq 5$, and there must exist some vertex not on the induced cycle adjacent to a vertex on the cycle, from which a path of length four can be induced in $\Rho$.

Case 3: If $k\geq 5$, there exists an induced path of length four in $\Rho$ by taking the subgraph induced by any four vertices on the cycle.

(b) \ Now assume $|V(\Rho)|=n+1\geq4$.  Then $\Rho$ must be a tree. If $\Rho=P_{n+1}$, then $\Rho$ clearly contains an induced path of length 4, so assume there exists at least one vertex $u$ with degree 3 or more. Suppose $\Rho \neq S_{n+1}$.  Then $|V(\Rho)|\geq 5$, and there exists a vertex $w$ that is not adjacent to $u$ but that has a common neighbor $v$ with $u$. Now choose any other vertex adjacent to $u$, say $t$, and induce the (unique) path of length four from $t$ to $w$.  Lemma \ref{indP_4} once again concludes the argument.
\end{proof}
\vspace{0.35cm}

The minimum number of distinct eigenvalues is well-studied for complete multipartite graphs \cite{AFM19}.  We tie one such result to the line graph of certain complete bipartite graphs.

\begin{proposition} \label{K_{2,n}}
If $\Rho = K_{2,n}$ for $n\geq 3$, then $\Rho$ is not a tight frame graph and $\Gamma = \SL(\Rho) = K_2 \cp K_n$ is a tight frame graph.
\end{proposition}

\begin{proof}
Corollary 6.5 of \cite{AAC13} is the first statement.  Corollary 6.8 of \cite{AAC13} implies the second statement.  As a direct demonstration, we note that
\[ F = \left[ J_n - I_n \quad J_n - (n-1)I_n \right] \]
is a tight frame associated with $\Gamma$.
\end{proof}
\vspace{0.35cm}

In general, $K_{m,n}$ is a tight frame graph if and only if $m=n$ (\cite{AAC13}, Corollary 6.5).  For $\SL(K_{m,n}) = K_m \cp K_n$, it is known (\cite{BHP18}, Proposition 3.1) that $q(\SL(K_{m,n})) \leq 3$.  To the best of the authors' knowledge, however, it remains open whether or not the Cartesian product of two complete graphs is a tight frame graph.  Note that every two non-adjacent vertices in $\SL(K_{m,n})$ have exactly two common neighbors \cite{H64}, so $\SL(K_{m,n})$ satisfies the necessary condition in Proposition \ref{neighbor} for being a tight frame graph.

We conclude by noting the limitations of the line graph approach.  Line graphs are characterized as those graphs that do not contain any of nine forbidden graphs as induced subgraphs \cite{B70}.  For reference, we list these graphs in Figure \ref{forbiddensubgraphs}, using Beineke's numbering, G1 through G9. 

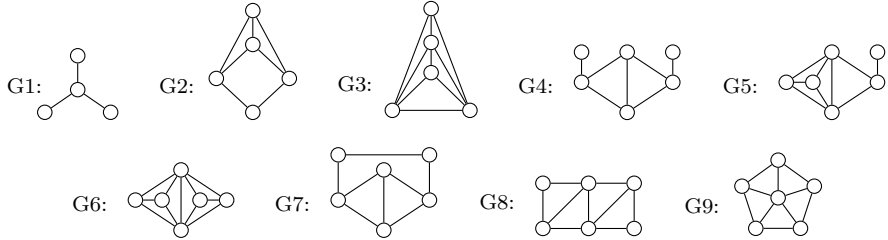
\begin{figure}[ht]
    \centering
        
        \begin{tikzpicture}
        \begin{scope}[every node/.style={circle,draw,inner sep=0pt, minimum size=\mysmnodesz} ]
            \node (v1) at (0,.05) {};
            \foreach \x in {2,...,4}{
                \pgfmathparse{90+\x*360/3}
                \node[draw,circle,inner sep=0] (v\x) at (\pgfmathresult:.5cm) {};} 
            \node[draw=none] (v5) at (-.7,.1) {\scriptsize G1:};
        \end{scope}
            \path (v1) edge node {} (v2);
            \path (v1) edge node {} (v3);
            \path (v1) edge node {} (v4);
        \end{tikzpicture}
        \quad
        \begin{tikzpicture}
        \begin{scope}[every node/.style={circle,draw,inner sep=0pt, minimum size=\mysmnodesz}]
        \node (v1) at (0,.45) {};
        \node (v2) at (0,.9) {};
        \node (v3) at (.48,0) {};
        \node (v4) at (-.48,0) {};
        \node (v5) at (0, -.45) {};
        \node[draw=none] (v6) at (-1,-.1) {\scriptsize G2:   };
        \end{scope}
        \path (v1) edge node {} (v2);
        \path (v1) edge node {} (v3);
        \path (v1) edge node {} (v4);
        \path (v2) edge node {} (v3);
        \path (v2) edge node {} (v4);
        \path (v5) edge node {} (v3);
        \path (v5) edge node {} (v4);
        \end{tikzpicture}
        \quad
        \begin{tikzpicture}
        \begin{scope}[every node/.style={circle,draw,inner sep=0pt, minimum size=\mysmnodesz}]
            \node (v1) at (-.51,0) {};
            \node (v2) at (.51,0) {};
            \node (v3) at (0,.5) {};
            \node (v4) at (0,.9) {};
            \node (v5) at (0,1.35){};
            \node[draw=none] (v6) at (-1,.35){\scriptsize G3:   };
        \end{scope}
            \path (v1) edge node {} (v2);
            \path (v1) edge node {} (v3);
            \path (v2) edge node {} (v3);
            \path (v4) edge node {} (v2);
            \path (v4) edge node {} (v3);
            \path (v4) edge node {} (v1);
            \path (v5) edge node {} (v2);
            \path (v5) edge node {} (v1);
            \path (v5) edge node {} (v4);
        \end{tikzpicture}
        \quad
        \begin{tikzpicture}
        \begin{scope}[every node/.style={circle,draw,inner sep=0pt, minimum size=\mysmnodesz}]
            \node (v1) at (-.6,0) {};
            \node (v2) at (.6,0) {};
            \node (v3) at (0,.4) {};
            \node (v4) at (0,-.4) {};
            \node (v5) at (-.6,.4){};
            \node (v6) at (.6,.4){};
            \node[draw=none] (v7) at (-1.2,-.05) {\scriptsize G4:   };
        \end{scope}
            \path (v1) edge node {} (v3);
            \path (v1) edge node {} (v4);
            \path (v1) edge node {} (v5);
            \path (v3) edge node {} (v2);
            \path (v4) edge node {} (v2);
            \path (v6) edge node {} (v2);
            \path (v3) edge node {} (v4);
        \end{tikzpicture}
        \quad
        \begin{tikzpicture}
        \begin{scope}[every node/.style={circle,draw,inner sep=0pt, minimum size=\mysmnodesz}]
            \node (v1) at (-.6,0) {};
            \node (v2) at (.6,0) {};
            \node (v3) at (0,.4) {};
            \node (v4) at (0,-.4) {};
            \node (v5) at (-.25,0){};
            \node (v6) at (.6,.4){};
            \node[draw=none] (v7) at (-1.2,-.05) {\scriptsize G5:   };
        \end{scope}
            \path (v1) edge node {} (v3);
            \path (v1) edge node {} (v4);
            \path (v1) edge node {} (v5);
            \path (v3) edge node {} (v5);
            \path (v4) edge node {} (v5);
            \path (v3) edge node {} (v2);
            \path (v4) edge node {} (v2);
            \path (v6) edge node {} (v2);
            \path (v3) edge node {} (v4);
        \end{tikzpicture}
        \\
        \vspace{3mm}
        \begin{tikzpicture}
        \begin{scope}[every node/.style={circle,draw,inner sep=0pt, minimum size=\mysmnodesz}]
            \node (v1) at (-.6,0) {};
            \node (v2) at (.6,0) {};
            \node (v3) at (0,.4) {};
            \node (v4) at (0,-.4) {};
            \node (v5) at (-.25,0){};
            \node (v6) at (.25,0){};
            \node[draw=none] (v7) at (-1.2,-.05) {\scriptsize G6:   };
        \end{scope}
            \path (v1) edge node {} (v3);
            \path (v1) edge node {} (v4);
            \path (v1) edge node {} (v5);
            \path (v3) edge node {} (v5);
            \path (v4) edge node {} (v5);
            \path (v3) edge node {} (v2);
            \path (v4) edge node {} (v2);
            \path (v6) edge node {} (v2);
            \path (v6) edge node {} (v3);
            \path (v6) edge node {} (v4);
            \path (v3) edge node {} (v4);
        \end{tikzpicture}
         \quad
        \begin{tikzpicture}
        \begin{scope}[every node/.style={circle,draw,inner sep=0pt, minimum size=\mysmnodesz}]
            \node (v1) at (-.6,0) {};
            \node (v2) at (.6,0) {};
            \node (v3) at (0,.4) {};
            \node (v4) at (0,-.4) {};
            \node (v5) at (-.6,.6){};
            \node (v6) at (.6,.6){};
            \node[draw=none] (v7) at (-1.2,-.05) {\scriptsize G7:   };
        \end{scope}
            \path (v1) edge node {} (v3);
            \path (v1) edge node {} (v4);
            \path (v1) edge node {} (v5);
            \path (v3) edge node {} (v2);
            \path (v4) edge node {} (v2);
            \path (v6) edge node {} (v2);
            \path (v3) edge node {} (v4);
            \path (v5) edge node {} (v6);
        \end{tikzpicture}
        \quad
        \begin{tikzpicture}
        \begin{scope}[every node/.style={circle,draw,inner sep=0pt, minimum size=\mysmnodesz}]
            \node (v1) at (-.6,0) {};
            \node (v2) at (0,0) {};
            \node (v3) at (.6,0) {};
            \node (v4) at (-.6,.6) {};
            \node (v5) at (0,.6){};
            \node (v6) at (.6,.6){};
            \node[draw=none] (v7) at (-1.2,.35) {\scriptsize G8:   };
        \end{scope}
            \path (v1) edge node {} (v2);
            \path (v2) edge node {} (v3);
            \path (v1) edge node {} (v4);
            \path (v1) edge node {} (v5);
            \path (v2) edge node {} (v5);
            \path (v2) edge node {} (v6);
            \path (v3) edge node {} (v6);
            \path (v4) edge node {} (v5);
            \path (v5) edge node {} (v6);
        \end{tikzpicture}
         \quad
        \begin{tikzpicture}
        \begin{scope}[every node/.style={circle,draw,inner sep=0pt, minimum size=\mysmnodesz}]
            \node (v1) at (0,0) {};
            \foreach \x in {2,...,6}{
                \pgfmathparse{90+\x*360/5}
                \node[draw,circle,inner sep=0] (v\x) at (\pgfmathresult:.5cm) {};} 
            \node[draw=none] (v7) at (-1,-.1) {\scriptsize G9:   };
        \end{scope}
           \path (v1) edge node {} (v2);
           \path (v1) edge node {} (v3);
           \path (v1) edge node {} (v4);
           \path (v1) edge node {} (v5);
           \path (v1) edge node {} (v6);
           \path (v2) edge node {} (v3);
           \path (v3) edge node {} (v4);
           \path (v4) edge node {} (v5);
           \path (v5) edge node {} (v6);
           \path (v6) edge node {} (v2);
        \end{tikzpicture}
    \caption{The nine forbidden subgraphs of \cite{B70}}
    \label{forbiddensubgraphs}
\end{figure}

Families known to be tight frame graphs are often not line graphs.  For example, recall that if $\Gamma$ is obtained from a complete graph ($n>3$) by deleting a single edge, then $\Gamma$ is a tight frame graph.  For $n=4$, $K_n - e$ is a line graph; however, for $n>4$, $K_n - e$ contains G3 as an induced subgraph and is therefore not a line graph.  The hypercube $Q_n$, defined recursively as $Q_1 = K_2$ and $Q_n = Q_{n-1} \cp K_2$, is another example:  $Q_n$ is a tight frame graph for each $n\geq 1$ (\cite{AAC13}, Corollary 6.9), but for $n\geq 3$, $Q_n$ contains the claw G1 as an induced subgraph and is therefore not a line graph.  

The {\em join} of two graphs $\Gamma$ and $\Lambda$, denoted $\Gamma \vee \Lambda$, is the graph with vertex set $V(\Gamma) \cup V(\Lambda)$ and edge set $E(\Gamma) \cup E(\Lambda) \cup \{ \{u,v\}: u\in V(\Gamma), \ v\in V(\Lambda) \}$.  For two connected graphs $\Gamma$ and $\Lambda$ on $n$ vertices, $\Gamma \vee \Lambda$ is a tight frame graph (\cite{MS16}, Theorem 5.2).  Most members of this large set of examples are not line graphs.

\begin{proposition}
If $\Gamma$ and $\Lambda$ are two connected graphs on $n\geq 3$ vertices, not both the complete graph, then $\Gamma \vee \Lambda$ is not a line graph.
\end{proposition}

\begin{proof}
Suppose $\Gamma \neq K_n$ and consider three vertices $u, v, w \in \Gamma$ and three vertices $a,b,c \in \Lambda$.

Case 1:  Assume the subgraph induced by $\{u,v,w\}$ is three isolated vertices.  Then $\{u,v,w,a\}$ induces the claw graph G1 as a subgraph in $\Gamma \vee \Lambda$.

Case 2:  Assume the subgraph induced by $\{u,v,w\}$ is the disjoint union of $P_2$  and an isolated vertex, where $w$ is the isolated vertex.  Due to the commutativity of the join operation, we need to consider only the cases when the subgraph induced by $\{a,b,c\}$ is the disjoint union of $P_2$ and an isolated vertex ($c$ being the isolated vertex), $P_3$, or $C_3$.  In the first case, $\{u,v,w,b,c\}$ induces G2, in the second case, $\{u,v,a,b,c\}$ induces G3, and in the third case, $\{v, w, a,b,c\}$ induces G3 in $\Gamma \vee \Lambda$.

Case 3:  Assume the subgraph induced by $\{u,v,w\}$ is $P_3$.  Again, the commutativity of join allows us to consider only the cases when the subgraph induced by $\{a,b,c\}$ is $P_3$ (with $b$ as the degree-two vertex) or $C_3$. In both cases, $\Gamma \vee \Lambda$ contains G3 as an induced subgraph, induced by $\{u,v,w,a,b\}$.

If every three vertices of $\Gamma$ induce a copy of $C_3$, then $\Gamma$ must be the complete graph on $n$ vertices, which contradicts our assumption.  Cases 1 through 3 are therefore exhaustive and always contain one of the nine forbidden induced subgraphs; hence, $\Gamma \vee \Lambda$ cannot be a line graph.
\end{proof}
\vspace{0.35cm}

Finally, we note that of the nine graphs listed in Figure \ref{forbiddensubgraphs}, G2, G3, and G6 are tight frame graphs 
(apply Lemma \ref{Chen_6.3} once to the tight frame graph $C_4$ and to the diamond graph of Example \ref{diamondframe} to obtain G2 and G3, respectively, and apply twice to the diamond graph to obtain G6).  Graphs G1, G4, G5, G7, and G8 fail to be tight frame graphs by Proposition \ref{neighbor}.  The wheel graph G9 is not a tight frame graph over symmetric real matrices (\cite{CGM14}, Example 6.13), and, to the best of the authors' knowledge, whether it is a tight frame graph over complex matrices remains an open question.  The occurrence of these nine forbidden graphs as induced subgraphs of tight frame graphs is, of course, guaranteed by Corollary \ref{min}.

\section*{Acknowledgments}
We thank Erich McAlister for his careful reading and helpful feedback, as well as the anonymous referees.


\end{document}